\documentclass[reqno,11pt]{amsart} 
\usepackage{amsmath,amstext,amsthm,amsfonts}
\usepackage{enumerate}

\usepackage{comment}
\usepackage[english]{babel}
\usepackage[colorlinks=true,urlcolor=blue,
citecolor=red,linkcolor=blue,linktocpage,pdfpagelabels,
bookmarksnumbered,bookmarksopen]{hyperref}

\usepackage[left=2.6cm,right=2.6cm,top=2.9cm,bottom=2.9cm]{geometry}

\usepackage{amsmath,amssymb,latexsym,soul,cite,mathrsfs}
\numberwithin{equation}{section}  
\usepackage{color,enumitem,graphicx}

\usepackage{cancel}

\usepackage[colorlinks=true,urlcolor=blue,
citecolor=red,linkcolor=blue,linktocpage,pdfpagelabels,
bookmarksnumbered,bookmarksopen]{hyperref}
\usepackage[english]{babel}

\newtheorem{definition}[]{Definition}
\newtheorem{example}[]{Example}
\newtheorem{theorem}[]{Theorem}
\newtheorem{corollary}[]{Corollary}
\newtheorem{lemma}[]{Lemma}
\newtheorem{proposition}[]{Proposition}

\DeclareMathOperator{\divg}{div}

\title[Some characterizations of compact Einstein-type manifolds]{Some characterizations of compact Einstein-type manifolds
}

\author[M. Andrade]{Maria Andrade}
\address[M. Andrade]{Departament of Mathematics, 
	Federal University of Sergipe
	\newline\indent 
	49100-000, Sao Cristov\~ao-SE, Brazil}
\email{\href{mailto: maria@mat.ufs.br}{ maria@mat.ufs.br}}

\author[A. P. de Melo]{Ana Paula de Melo}

\address[A. P. de Melo]{ Institute of Mathematics, 
	Federal University of Goi\'as
	\newline\indent 
	74690-900, Goiânia-GO, Brazil}
\email{\href{mailto: anapmelocosta@gmail.com}{anapmelocosta@gmail.com}}


\begin{document}

\maketitle
\begin{abstract}
In this work, we investigate the geometry and topology of compact Einstein-type manifolds with nonempty boundary. First, we prove a sharp boundary estimate, as consequence we obtain under certain hypotheses that the Hawking mass is bounded from bellow in terms of area. Then we give a topological classification for its boundary. Finally, we prove a gap result for a compact Einstein-type manifold with boundary.
\end{abstract}

\section{Introduction}
In the last decades, problems related to Einstein manifolds have become very present in theory, because there are many applications in mathematics and theorical physics. 
We recall that a Riemannian manifold $(M^n,g)$ is said to be Einstein if 
$Ric=\frac{R}{n}g,$ where $Ric$ and $R$ are Ricci and scalar curvatures, respectively. Throughout the text the dimension will be considered $n\geq 3$, unless explicitly mentioned. More recently, there have been some generalization about Einstein manifold (see e.g\cite{catino2016geometry} and \cite{leandro2021vanishing}). Here we use the following definition.

\begin{definition} \label{defiqe}A Riemannian manifold $(M^n, g)$, is called an Einstein-type manifold if there are smooth functions $f, h :M\to \mathbb{R}$ such that
\begin{eqnarray}\label{eq000}
fRic=\nabla^2f+hg,
\end{eqnarray}
where $f >0$ in int(M), $f=0$ on $\partial M$ and $\nabla^2$ is the Hessian. We denote $(M^n, g, f, h)$ Einstein-type manifold. 
\end{definition}


Using the equation \eqref{eq000}, it is easy to see that 
\begin{eqnarray}\label{eq001}
fR=\Delta f+nh,
\end{eqnarray}
here $\Delta$ is the Laplacian operator.

We observe that if $(M^n, g, f, h)$ is an Einstein-type manifold, then it satisfies the following equation:
\begin{equation}\label{eq002}
     \mathring{\nabla^2}f=f\mathring{Ric},
\end{equation}
where $\mathring{A}$ is the traceless tensor, i.e, $\mathring{A}= A-\dfrac{\text{trace}(A)}{n}g.$

It is possible to notice that the class of manifolds satisfying the Definition \ref{defiqe} generalizes several important examples:

\begin{example}
\begin{enumerate}
\item[a)] If we consider $h =0$ and $R =0,$ we obtain the Static vacuum Einstein equation \cite{hwang2016nonexistence}.
\item[b)] If we take $h=\dfrac{Rf}{n-1},$ we get the static vaccum equation with non null cosmological constant \cite{ambrozio2017static}.
\item[c)] If $h=\dfrac{\mu-\rho}{n-1}f,$ we have static perfect fluid, where $\mu=R/2$ and $\rho$ are, respectively, the density and the pressure smooth functions \cite{shen1997note}.
\item[d)] If we consider $h=\dfrac{Rf+1}{n-1}$,  we get the called Miao-Tam equation \cite{miao2011einstein}.
\item[e)] If we consider $h=\lambda f$, where $\lambda$ is a constant we obtain $(\lambda, n+1)$-Einstein manifold. Moreover, if we take $f=e^{-\phi},$ here $\phi$ is defined in the interior of $M$, then we obtain the well-known Bakry-Emery Ricci Tensor  \cite{freitas2020boundary}.
\end{enumerate}
 \end{example}
 
 For instance, Shen \cite{shen1997note} used some ideias from general relativity and proved a Robinson-type identity obtaining results about a Fisher-Marsden conjecture. Hwang et al. \cite{hwang2016nonexistence} proved that there are no multiple black holes in an $n$-dimensional static vacuum space-time having weakly harmonic curvature unless the Ricci curvature is trivial. In \cite{ambrozio2017static}, Ambrozio proved, among others interesting results,  some classification results for
compact static three-manifolds with positive scalar curvature.  Miao and Tam \cite{miao2011einstein} classified all Einstein or conformally flat metrics which are critical points of volume in certain space. Barros and Gomes \cite{barros2013compact} proved that a compact gradient generalized $m$-quasi-Einstein metric with constant scalar curvature must be isometric to a standard Euclidean sphere $\mathbb{S}^n$ with the potential $f$ well determined.  Later, Coutinho et al. \cite{coutinho2019static} studied the geometry of static perfect fluid space-time metrics on compact manifolds with boundary yielding a gap result for this space. Then, Freitas and Santos \cite{freitas2020boundary} investigated about generalized compact Einstein manifolds, giving some topological classifications for its boundary. Leandro \cite{leandro2021vanishing} considered an Einstein-type equation which generalizes important geometric equations, he proved a result, among others, about the nonexistence of multiple black holes in static spacetimes.
 
 Motivated by \cite{coutinho2019static} and \cite{leandro2021vanishing}, we obtain some characterization results for the Einstein-type manifolds with dimension $n\geq 3$, compact and with boundary. Our first result is the following:
 
\begin{theorem} \label{einsedp}Let $(M^n,g,f, h)$ be a compact Einstein-type manifold with boundary. If $(M^n, g)$ is Einstein, then $f$ satisfies the following differential equation
$$\nabla^2f=\left(-\dfrac{R}{n(n-1)}f+C\right)g,$$
where $C$ is a constant. In particular, if $h$ or $f$ is not constant, then, for $R\geq 0,$ $(M^n,g)$ is isometric to a geodesic ball on a sphere $\mathbb{S}^n$.
\end{theorem}

An interesting result obtained by Boucher, Gibbons and Horowitz \cite{boucher1984uniqueness}, and by Shen \cite{shen1997note} showed that a boundary $\partial M$ of a compact $3$-dimensional oriented static manifold with connected boundary and positive scalar curvature equals to $6$ must be a $2$-sphere whose area satisfies the inequality $$|\partial M|\leq 4\pi.$$ 
The equality occurs if and only if $M^3$ is isometric to a standard hemisphere. Inspired by this result and its natural extension for static perfect fluid space-time with boundary $\partial M$ \cite{coutinho2019static} and generalized $(\lambda,n+m)$-Einstein manifolds \cite{freitas2020boundary}, we obtain the next result for Einstein-type manifold. 

\begin{theorem}\label{comparevol}
Let $(M^n, g,f,h)$ be a compact, oriented Einstein-type manifold with boundary $\partial M$ such that $Ric^{\partial M}\geq \dfrac{R^{\partial M}}{n-1}g_{\partial M}$ with ${\text{inf }}R^{\partial M}>0$. Suppose that either:
\begin{enumerate}
    \item The scalar curvature $R$ is a positive constant, or
    \item $h\geq \dfrac{f}{n}R$ and $R_{min}>0.$
\end{enumerate}
Then,
\begin{eqnarray*}\label{areavol}
|\partial M| \leq \left(\dfrac{n(n-1)}{R_{min}+K(n,H)}\right)^{\frac{n-1}{2}}w_{n-1},
\end{eqnarray*}
where $Ric^{\partial M}$ and $R^{\partial M}$ are Ricci and scalar curvatures on $\partial M,$ respectively, $|\partial M|$ is the area of $\partial M$, $K(n,H)=\dfrac{(n-1)n}{|\partial M|}\displaystyle\int_{\partial M}H^2dS,$ $H$ is the mean curvature of $\partial M,$ $R_{min}$ is the minimum value of $R$ on $M^n$ and $w_{n-1}$ denotes the volume of the standard unit sphere. In particular,
\begin{eqnarray}\label{areavol1}
|\partial M| \leq \left(\dfrac{n(n-1)}{R_{min}}\right)^{\frac{n-1}{2}}w_{n-1}.
\end{eqnarray}
The equality holds in \eqref{areavol1} if and only if $(M^n,g)$ is an Einstein manifold with totally geodesic boundary. In this case, $(M^n,g)$ is isometric to a geodesic ball on a sphere $\mathbb{S}^n.$
\end{theorem}

We remember that in the proof of the positive mass theorem by Schoen and Yau \cite{schoen1979proof} a crucial point is the study of minimal surfaces in a certain space $M^3$ that is important in general relativity.
In 1973, motivated by a physical problem, Penrose conjectured that 
\begin{eqnarray*}\label{mass}
16\pi m^2\geq |\Sigma|,
\end{eqnarray*}
where $\Sigma$ is a minimal surface in $M^3,$ $m$ is the so-called ADM mass and $|\Sigma|$ is the area of $\Sigma$. In this way, Huisken and Ilmanen \cite{huisken2001inverse} proved this conjecture under certain hypotheses.
To prove this result they used the Hawking quasi-local mass of a $2$-surface defined by
\begin{eqnarray}\label{massH}
{\bf{m}}_H(\Sigma)=\sqrt{\dfrac{|\Sigma|}{16\pi}}\left(1-\dfrac{1}{16\pi}\displaystyle\int_{\Sigma}H^2dS\right).
\end{eqnarray}
This quantity has been proposed as a quasi-local measure for the strength of the gravitational field \cite{christodoulou71some}. Hawking observed that it approaches the ADM mass for large coordinates spheres.

Inspired by these works and the Theorem \ref{comparevol}, we obtain an interesting application, which shows under certain hypotheses that the Hawking mass is bounded from bellow in terms of area.

\begin{corollary}\label{comparevol1}
Let $(M^3, g,f,h)$ be a compact, oriented Einstein-type manifold with boundary $\partial M$ closed, two-sided surface such that $Ric^{\partial M}\geq \dfrac{R^{\partial M}}{2}g_{\partial M}$  with ${\text{inf }}R^{\partial M}>0$. Suppose that either:
\begin{enumerate}
    \item The scalar curvature $R$ is a positive constant, or
    \item $h\geq \dfrac{f}{3}R$ and $R_{min}>0$.
\end{enumerate}
Then,
\begin{eqnarray}\label{massnH}
{\bf m}_H(\partial M)\geq\dfrac{1}{96\pi}\left(\sqrt{\dfrac{|\partial M|}{16\pi}}\right)(72\pi+R_{min}|\partial M|).
\end{eqnarray}
The equality holds in \eqref{massnH} if and only if $(M^3,g)$ is an Einstein manifold with totally geodesic boundary and in particular the Hawking mass satisfies
$${\bf{m}}_H(\partial M)=\sqrt{\dfrac{|\partial M|}{16\pi}}.$$
In this case, $(M^3,g)$ is isometric to a geodesic ball on a sphere $\mathbb{S}^3.$ 
\end{corollary}

Moreover, using the Gauss-Bonnet theorem we obtain the following topological characterization for the boundary in compact, oriented Einstein-type manifolds $(M^3,g,f,h).$ 

\begin{corollary}\label{ETM3}Let $(M^3,g,f,h)$ be a compact, oriented Einstein-type manifold with boundary. Suppose that either 
\begin{enumerate}
    \item The scalar curvature $R$ is constant, or
    \item $h\geq \dfrac{f}{3}R.$
\end{enumerate}
Then, 
$$\chi(\partial M)\geq \dfrac{R_{min}}{12 \pi}|\partial M|,$$
where $\chi(\partial M)$ is the Euler characteristic of $\partial M.$ Moreover, the equality holds if and only if $(M^3, g)$ is an Einstein manifold with totally geodesic boundary. In particular, if $R_{min}>0,$ then $\partial M$ is topologically a $2$-sphere.
\end{corollary}

Furthermore, we can prove that the geodesic ball on a sphere $\mathbb{S}^3$ is the unique compact Einstein-type manifold with positive constant scalar curvature such that the norm of the without trace Einstein tensor $|\mathring{Ric}|$ lies in the interval $\left[0,\frac{\sqrt{6}}{12}(\frac{6h}{f}-R)\right).$ More precisely, we prove that following result.
\begin{theorem}\label{geodball} Let $(M^3,g,f,h)$ be a compact, oriented Einstein-type manifold with positive constant scalar curvature satisfying the gap condition
\begin{eqnarray}\label{eqn005}
|\mathring{Ric}|<\dfrac{\sqrt{6}}{12}\left(\dfrac{6h}{f}-R\right).
\end{eqnarray}
Then $M^3$ is isometric to a geodesic ball on a sphere $\mathbb{S}^3.$
\end{theorem}

\section{Background and proofs}
Now, we start this section with a Lemma for an Einstein-type manifold $(M^n,g,f,h)$ that shows a necessary and sufficient condition in terms of $f$ and $h$ for $(M^n,g)$ has constant scalar curvature. Then we present the proofs of our main results.
 
\begin{lemma}\label{lemaRconst}Let $(M^n, g,f,h)$ be an Einstein-type manifold. The scalar curvature of $M^n$ is constant if and only if $Rf-(n-1)h$ is constant.
\end{lemma}
\begin{proof} To prove this result, first we use the second Bianchi identity given by $\divg Ric= \frac{1}{2}\nabla R$ to obtain that
\begin{eqnarray}\label{eqn003}
\divg(f\mathring{Ric})&=&\divg\left(fRic-f\frac{R}{n}g\right)\nonumber\\
                      &=&f\divg(Ric)+Ric(\nabla f)-\frac{1}{n}\nabla(fR)\nonumber\\
                      &=&\dfrac{f}{2}\nabla R+Ric(\nabla f)-\frac{R}{n}\nabla(f)-\frac{f}{n}\nabla(R)\nonumber\\
                      &=&\dfrac{(n-2)}{2n}f\nabla R+Ric(\nabla f)-\frac{R}{n}\nabla f 
\end{eqnarray}
Second, using that $\divg \nabla^2f=Ric(\nabla f)+\nabla \Delta f,$ we deduce
\begin{eqnarray}\label{eq004}
\divg(\mathring{\nabla^2}f)=Ric(\nabla f) +\frac{n-1}{n}\nabla\Delta f.
\end{eqnarray}
Next, by equations \eqref{eq001}, \eqref{eq002}, \eqref{eqn003} and \eqref{eq004}, we get
\begin{eqnarray}\label{eq004i}
\dfrac{n-1}{n}\nabla \Delta f=\dfrac{n-2}{2n}f\nabla R-\dfrac{R}{n}\nabla f.
\end{eqnarray}
Finally, we deduce that
$$\dfrac{1}{2}f\nabla R=\nabla(Rf-(n-1)h).$$

The proof is finished.
\end{proof}

To prove Theorem \ref{einsedp}, we need the following Proposition from \cite{coutinho2019static}, which is a consequence of Obata's work \cite{obata1962certain} and Reilly's theorem \cite{reilly1980geometric}.

\begin{proposition}[\cite{coutinho2019static}, Proposition 1]\label{isometric} Let $(M^n, g, f)$ be a compact Einstein manifold with positive scalar curvature and $f$ a smooth function on $M^n$ satisfying \eqref{eq002}.
\begin{enumerate}
\item If $\partial M$ is empty, then $M^n$ is isometric to a round sphere $\mathbb{S}^n.$
\item If $\partial M$ is connected non-empty and $f|_{\partial M}$ is constant, then $M^n$ is isometric to a geodesic ball on a sphere $\mathbb{S}^n.$
\end{enumerate}
\end{proposition}

\begin{proof}[Proof of Theorem \ref{einsedp}] If $(M^n,g)$ is Einstein, $n\geq 3,$ then $R$ is constant. So, by Lemma \ref{lemaRconst}, we obtain that 
\begin{equation}\label{fhconst}
Rf-(n-1)h=c,
\end{equation}
where $c$ is a constant. Since $(M^n,g)$ is Einstein and $f$ satisfies \eqref{eq002}, we deduce that $\nabla^2f=\frac{\Delta f}{n}g.$ Using \eqref{eq001} and \eqref{fhconst}, we obtain that $f$ satisfies the following differential equation
$$\nabla^2f=\left(-\dfrac{R}{n(n-1)}f+C\right)g,$$
here $C=c/(n-1).$

Now, from \eqref{fhconst} we infer that $f$ is constant if and only if $h$ is constant. If $R=0,$ then $f$ is constant, but this does not occurs. Thus, the scalar curvature is positive, i.e., $R>0$. Since $f=0,$ on $\partial M,$ and $f$ satisfies the equation \eqref{eq002}, then by Proposition \ref{isometric} we conclude that $M^n$ is isometric to a geodesic ball on a sphere $\mathbb{S}^n.$ 
\end{proof}



\subsection{Einstein-type manifolds with boundary}

In this subsection, we study compact Einstein-type manifold with boundary.
By definition  $f=0$ on $\partial M,$ then $f$ does not change of sign on $\partial M.$ In particular, $|\nabla f|\neq 0$ on $\partial M$ and we can consider the normal vector on $\partial M$ defined by $\nu=-\dfrac{\nabla f}{|\nabla f|}.$ 
Since $f=0$ on $\partial M,$ then by equation \eqref{eq002}, we obtain that $\nabla^2 f= \dfrac{\Delta f}{n}g,$ on $\partial M.$ This implies that 
\begin{eqnarray*}
X(|\nabla f|^2)= 2\langle \nabla_X\nabla f, \nabla f\rangle=2\nabla^2f(X,\nabla f)=2\dfrac{\Delta f}{n}g(X,\nabla f)=0,
\end{eqnarray*}
where $X\in \mathfrak{X}(\partial M).$ This proves that $|\nabla f|$ is constant and non-null on $\partial M.$ Now, we consider an orthonormal frame $\{e_1,\cdots, e_{n-1},e_n=\nu\}$ on $\partial M.$ From equations \eqref{eq001}, \eqref{eq002} and $f=0$ on $\partial M,$ we infer that
\begin{eqnarray}\label{eqhes}
\nabla^2f=\dfrac{\Delta f}{n}g=-hg.
\end{eqnarray}
We denote by $\alpha_{ab}$ the second fundamental form, where $1\leq a,b\leq n-1.$ Then by definition of $\alpha_{ab}$ and using \eqref{eqhes}, we obtain
$$\alpha_{ab}=\langle \nabla_{e_a}\nu,e_b\rangle=-\dfrac{1}{|\nabla f|}\langle \nabla_{e_a}\nabla f,e_b\rangle=-\dfrac{1}{|\nabla f|}\nabla_a\nabla_bf=-\dfrac{\Delta f}{n|\nabla f|}g_{ab}=\dfrac{h}{|\nabla f|}g_{ab}.$$

This shows that $\partial M$ is totally umbilical with mean curvature $H =\dfrac{h}{|\nabla f|}.$ From Gauss equation, we deduce 
\begin{eqnarray*}
R^{\partial M}_{abcd}&=&R_{abcd}-\alpha_{ad}\alpha_{bc}+\alpha_{ac}\alpha_{bd}.
\end{eqnarray*}
This implies that
$$R_{ac}^{\partial M}=R_{ac}-R_{ancn}+\dfrac{h^2}{|\nabla f|^2}(n-2),$$
and finally, we obtain that the scalar curvature on $\partial M$ is given by
\begin{eqnarray*}
R^{\partial M}=R-2R_{nn}+\dfrac{h^2}{|\nabla f|^2}(n-1)(n-2).
\end{eqnarray*}
Which is equivalent to
\begin{eqnarray}\label{eqgauss}
R_{nn}=\dfrac{R-R^{\partial M}}{2}+\dfrac{(n-1)(n-2)}{2}H^2.
\end{eqnarray}

In this way, we obtain the following Lemma.

\begin{lemma}\label{propintbound}
Let $(M^n, g,f,h)$ be a compact, oriented Einstein-type manifold. Then
\begin{eqnarray*}
\displaystyle\int_{\partial M}|\nabla f|R^{\partial M} dS - (n-1)(n-2)\displaystyle\int_{\partial M}|\nabla f|H^2 dS&=&2\displaystyle\int_{M}f|\mathring{Ric}|^2\\
&&-\frac{(n-2)}{n}\displaystyle\int_M R\Delta f dV.
\end{eqnarray*}
\end{lemma}

\begin{proof}
It is well known that $\divg(\mathring{Ric}(\nabla f))=\divg\mathring{Ric}(\nabla f)+\langle \mathring{Ric}, \nabla^2f \rangle.$ Using that $\nabla^2 f = f Ric-hg$ and the second identity of Bianchi $\divg Ric = \frac{1}{2}\nabla R,$ we deduce
$$\divg(\mathring{Ric}(\nabla f))=\dfrac{n-2}{2n}\langle \nabla R, \nabla f\rangle+f|\mathring{Ric}|^2.$$
Integrating the last equation over $M$ and using the Stokes's theorem, we obtain
\begin{eqnarray}\label{eq007}
\displaystyle\int_M\divg(\mathring{Ric}(\nabla f))dV\!\!\!\!&=&\!\!\!\!\displaystyle\int_Mf|\mathring{Ric}|^2dV+\dfrac{n-2}{2n}\displaystyle\int_M\langle \nabla R, \nabla f\rangle dV\nonumber\\
\!\!\!\!&=&\!\!\!\!\displaystyle\int_M f|\mathring{Ric}|^2 dV+\dfrac{n-2}{2n}\left(\displaystyle\int_{\partial M}R\langle \nabla f, \nu \rangle dS-\displaystyle\int_{M}R\Delta f dV\right)\nonumber\\
\!\!\!\!&=&\!\!\!\!\displaystyle\int_M f|\mathring{Ric}|^2dV-\dfrac{n-2}{2n}\left(\displaystyle\int_{\partial M}R |\nabla f| dS +\displaystyle\int_{M}R\Delta f dV\right).
\end{eqnarray}
By one hand, we infer 
\begin{eqnarray}\label{eq008}
\displaystyle\int_M\divg(\mathring{Ric}(\nabla f))dV&=&\displaystyle\int_{\partial M}\langle \mathring{Ric}(\nabla f), \nu\rangle dS\nonumber\\
&=&-\displaystyle\int_{\partial M}|\nabla f| \mathring{Ric}(\nu,\nu)dS\nonumber\\
&=&-\displaystyle\int_{\partial M}|\nabla f|\left(R_{nn}-\dfrac{R}{n}\right)dS.
\end{eqnarray}
By other hand, combining Eqs. \eqref{eqgauss}, \eqref{eq007} and \eqref{eq008}, we get
\begin{eqnarray*}
\displaystyle\int_Mf|\mathring{Ric}|^2dV-\dfrac{n-2}{2n}\displaystyle\int_MR\Delta fdV&=&-\displaystyle\int_{\partial M}|\nabla f|\left(R_{nn}-\dfrac{R}{n}\right)dS\\
&& + \dfrac{n-2}{2n}\displaystyle\int_{\partial M}R|\nabla f|dS\\
&=&\dfrac{1}{2}\displaystyle\int_{\partial M}|\nabla f|R^{\partial M} dS\\
&&-\dfrac{(n-1)(n-2)}{2}\displaystyle\int_{\partial M}|\nabla f|H^2dS. 
\end{eqnarray*}

\end{proof}

After,  motivated by Proposition 3 in \cite{coutinho2019static} in the context of  static perfect fluid space-time, we prove a result in an Einstein-type manifold. More precisely, we obtain the following.
\begin{proposition}\label{integralH}
Let $(M^n, g,f,h)$ be a compact, oriented Einstein-type manifold. Suppose that either:
\begin{enumerate}
    \item The scalar curvature $R$ is constant, or
    \item $h\geq \dfrac{f}{n}R.$
\end{enumerate}
Then,
$$\int_{\partial M}R^{\partial M} dS\geq \dfrac{2}{k}\int_{M}f|\mathring{Ric}|^2dV+\dfrac{n-2}{n}R_{min}|\partial M|+(n-1)(n-2)\int_{\partial M}H^2dS,$$
where $k=|\nabla f|$ on $\partial M,$ $R_{min}$ is the minimum value of $R$ on $M^n$ and $|\partial M|$ is the area of $\partial M.$
In particular, we obtain
$$\displaystyle\int_{\partial M}R^{\partial M} dS\geq \dfrac{n-2}{n}R_{min}|\partial M|,$$
and the equality occurs if and only if $(M^n,g)$ is Einstein and $\partial M$ is totally geodesic.
\end{proposition}

\begin{proof}
If $R$ is constant, then
$$\displaystyle\int_M R\Delta f dV=R \displaystyle\int_M \Delta f dV= R\displaystyle\int_{\partial M}\langle \nabla f, \nu \rangle= -R\displaystyle\int_{\partial M}|\nabla f|dS = - R k |\partial M|.$$

Now using the Lemma \ref{propintbound}, we obtain
\begin{eqnarray}\label{eqrmin}
k\displaystyle\int_{\partial M}R^{\partial M} dS &=&(n-1)(n-2)k\displaystyle\int_{\partial M}H^2dS+2\displaystyle\int_{M}f|\mathring{Ric}|^2dV+\dfrac{n-2}{n}Rk|\partial M|\nonumber\\
&\geq& 2\displaystyle\int_{M}f|\mathring{Ric}|^2dV+\dfrac{n-2}{n}Rk|\partial M|\nonumber\\
&\geq&\dfrac{n-2}{n}Rk|\partial M|.
\end{eqnarray}

By hypothesis $h\geq \dfrac{f}{n}R,$ this implies that $R\Delta f\leq R_{min} \Delta f.$ We deduce

$$\displaystyle\int_{M}R\Delta f\leq R_{min}\displaystyle\int_{M}\Delta f dV= - R_{min} k |\partial M|.$$

Again, by Lemma \ref{propintbound}, we infer
$$k\displaystyle\int_{\partial M}R^{\partial M} dS \geq(n-1)(n-2)k\displaystyle\int_{\partial M}H^2dS+2\displaystyle\int_{M}f|\mathring{Ric}|^2dV+\dfrac{n-2}{n}R_{min}k|\partial M|.$$

In particular, 
\begin{eqnarray}\label{eq009}
\displaystyle\int_{\partial M} R^{\partial M} \geq \dfrac{n-2}{n} R_{min}|\partial M|.
\end{eqnarray}

If occurs the equality in \eqref{eqrmin} or in \eqref{eq009}, then $\mathring{Ric}=0,$ i.e $(M^n,g)$ is an Einstein manifold and $H=0$ and since $\partial M$ is totally umbilical, then the result follows. 

If $h=fR/n,$ then $\Delta f = 0$ and  because $M$ is compact, we conclude that $f$ is constant. This implies that
$$f\displaystyle\int_{M}|\mathring{Ric}|^2dv+(n-1)(n-2)\displaystyle\int_{\partial M}H^2dS=0.$$
By definition $f> 0$ in int(M), then $\mathring{Ric}=0$ and $H=0,$ this finishes the proof.
\end{proof}

\begin{proof}[Proof of Corollary \ref{ETM3}] From Proposition \ref{integralH}, we obtain
$$\displaystyle\int_{\partial M}R^{\partial M}dS\geq\dfrac{1}{3}R_{min}|\partial M|.$$
Now, using the Gauss-Bonnet theorem, we infer  
$$4\pi\chi(\partial M)=2\displaystyle\int_{\partial M}KdS=\displaystyle\int_{\partial M}R^{\partial M}dS\geq \dfrac{1}{3}R_{min}|\partial M|,$$
where $K$ is the Gaussian curvature of $\partial M.$ Thus, 
$$\chi(\partial M)\geq \dfrac{1}{12\pi}R_{min}|\partial M|.$$
In particular, if $R_{min}>0,$ then $\chi(\partial M)>0.$ So, in this case $\partial M$ is topologically a $2$-sphere.
\end{proof}

Now we are ready to prove the Theorem \ref{comparevol}.
\begin{proof}[Proof of Theorem \ref{comparevol}]
Since ${\text inf} R^{\partial M}> 0$ and $Ric^{\partial M}\geq \dfrac{R^{\partial M}}{n-1}g_{\partial M},$ then there exists $\delta> 0$ such that 
$${\text{inf}}\{Ric^{\partial M}(V,V); \ V \in \ T\partial M,\ |V| =1\}=(n-2)\delta.$$
So,
\begin{eqnarray*}\label{rici}
Ric^{\partial M}\geq (n-2)\delta.
\end{eqnarray*}
Using Bonnet-Myers theorem, we deduce that the diameter of $\partial M$ satisifies $diam(\partial M)\leq \frac{\pi}{\sqrt{\delta}}.$ Now, from Bishop-Gromov theorem, we infer that
$$|\partial M|\leq vol(B^{\partial M}_{\frac{\pi}{\sqrt{\delta}}})\leq \delta^{-\frac{n-1}{2}}w_{n-1}.$$
We observe that there exists an unit vector field $V$ such that $Ric^{\partial M}(V,V)=(n-2)\delta$. We deduce that $(n-2)\delta\geq \frac{R^{\partial M}}{ n-1}.$ This implies that 
$$(n-1)(n-2)w_{n-1}^{\frac{2}{n-1}}\geq R^{\partial M} |\partial M|^{\frac{2}{n-1}}.$$
Integrating this expression over $\partial M,$ we obtain
\begin{eqnarray}\label{eqcomp}
(n-1)(n-2)w_{n-1}^{\frac{2}{n-1}}\geq  |\partial M|^{-\frac{(n-3)}{n-1}}\displaystyle\int_{\partial M}R^{\partial M}dS.
\end{eqnarray}
From Proposition \ref{integralH} and \eqref{eqcomp}, we deduce that
\begin{eqnarray*}
(n-1)w_{n-1}^{\frac{2}{n-1}}\geq \dfrac{|\partial M|^{\frac{2}{n-1}}}{n}\left(R_{min}+\dfrac{(n-1)n}{|\partial M|}\displaystyle\int_{\partial M}H^2dS\right).
\end{eqnarray*}
Thus,
\begin{eqnarray}\label{eqarea}
|\partial M|\leq\left(\frac{n(n-1)}{R_{min}+K(H,n)}\right)^{\frac{n-1}{2}}w_{n-1} \leq\left(\frac{n(n-1)}{R_{min}}\right)^{\frac{n-1}{2}}w_{n-1},
\end{eqnarray}
where $K(H,n)=\frac{n(n-1)}{|\partial M|}\displaystyle\int_{\partial M}H^2dS.$

Moreover, the equality holds in \eqref{eqarea} if and only if $(M^n,g)$ is Einstein with totally geodesic boundary. Finally, using the Proposition \ref{isometric}, we conclude that $(M^n,g)$ is isometric to a geodesic ball on a sphere $\mathbb{S}^n.$
\end{proof}

\begin{proof}[Proof of Corollary \ref{comparevol1}] By Theorem \ref{comparevol} and definition of Hawking quasi-local mass \eqref{massH}, we obtain the result.
\end{proof}

In \cite{coutinho2019static} was obtained a Böchner type formula for Riemannian manifolds that satisfies the equation \eqref{eq002}, was proved  the following result.

\begin{lemma}[\cite{coutinho2019static}, Theorem 2]\label{divf} Let $(M^3,g)$ be a Riemannian manifold and $f$ a smooth function on $M^n$ satisfying $f\mathring{Ric}=\mathring{\nabla}^2f$. Then it holds
\begin{eqnarray*}
\dfrac{1}{2}div\left(f\nabla|\mathring{Ric}|^2-\dfrac{1}{2}f\mathring{Ric}(\nabla R)\right)&=& f|\nabla \mathring{Ric}|^2-\dfrac{1}{24}f|\nabla R|^2+f|C|^2\\
&+&\left(\dfrac{Rf}{2}-\Delta f\right)|\mathring{Ric}|^2+6ftr(\mathring{Ric}^3),
\end{eqnarray*}
where $|C|^2= C_{ijk}C^{ijk}$ and $C_{ijk}$ is the Cotton tensor.
\end{lemma}

In particular, we can show that the geodesic ball on a sphere $\mathbb{S}^3$ is the unique compact Einstein-type manifold with positive constant scalar curvature such that the norm of the without trace Einstein tensor $|\mathring{Ric}|$ lies in the interval $[0,\frac{\sqrt{6}}{12}(\frac{6h}{f}-R)).$ More precisely, we prove the Theorem \ref{geodball}.
\begin{proof}[Proof of Theorem \ref{geodball}]
Since the scalar curvature $R$ is constant, then by Lemma \ref{divf}, we obtain
\begin{equation}\label{eq006}
\dfrac{1}{2}div\left(f\nabla|\mathring{Ric}|^2\right)= f|\nabla \mathring{Ric}|^2+f|C|^2+\left(\dfrac{Rf}{2}-\Delta f\right)|\mathring{Ric}|^2+6ftr(\mathring{Ric}^3).
\end{equation}
Substituting \eqref{eq001} and the Okumura’s inequality $6tr(\mathring{Ric}^3)\geq -\sqrt{6}|\mathring{Ric}|^3$ in \eqref{eq006}, we infer that
\begin{equation*}
\dfrac{1}{2}div\left(f\nabla|\mathring{Ric}|^2\right)\geq f|\nabla \mathring{Ric}|^2+f|C|^2+\left(-\dfrac{Rf}{2}+3h-\sqrt{6}f|\mathring{Ric}|\right)|\mathring{Ric}|^2.
\end{equation*}

Integrating over $M,$ using that $f=0$ on $\partial M$ and the gap condition \eqref{eqn005}, we obtain
$$0\geq\displaystyle\int_{M}\left(f|\nabla \mathring{Ric}|^2+f|C|^2+\left(-\dfrac{Rf}{2}+3h-\sqrt{6}f|\mathring{Ric}|\right)|\mathring{Ric}|^2\right)dS\geq 0.$$
Again using \eqref{eqn005}, we deduce that $\mathring{Ric}=0,$ i.e., $(M^3,g)$ is an Einstein manifold. Finally, by Proposition \ref{isometric} we conclude the proof.
\end{proof}

\section*{Acknowledgment} 
The first author was partially supported by Brazilian National Council for Scientific and Technological Development (CNPq Grant 403349/2021-4) and FAPITEC/SE/Brazil. The second author was partially supported by Brazilian National Council for Scientific and Technological Development (CNPq Grant 403349/2021-4) and PROPG-CAPES. The authors are grateful to Professor Benedito Leandro for his valorous comments about this work.

\bibliographystyle{acm}

\bibliography{sample}

\end{document}